\documentclass[a4paper,12pt]{amsart}
\makeatletter
\@namedef{subjclassname@2020}{%
\textup{2020} Mathematics Subject Classification}
\makeatother
\usepackage[shortlabels]{enumitem}
\newtheorem{thm}{Theorem}[section]
\newtheorem{defn}[thm]{Definition}
\newtheorem{rem}[thm]{Remark}
\newtheorem{prop}[thm]{Proposition}
\newtheorem{cor}[thm]{Corollary}
\newtheorem{exam}[thm]{Example}

\newtheorem{lemma}[thm]{Lemma}
\newcommand{\eps}{\varepsilon}

\renewcommand{\Im}{\operatorname{\rm{Im}}}
\renewcommand{\Re}{\operatorname{\rm{Re}}}
\def\var{\varepsilon}
\def\C{\mathbb C}
\def\D{\mathbb D}
\def\R{\mathbb R}
\def\ds{\displaystyle}
\def\k{\kappa}
\def\d{\delta}
\def\a{\alpha}
\def\b{\beta}
\def\g{\gamma}
\def\s{\sigma}

\newenvironment{proof*}{\vskip 2mm\noindent {}}{\hfill $\Box$ \vskip 2mm}

\begin{document}

\numberwithin{equation}{section}

\title[Visibility of Kobayashi geodesics in convex domains]
{Visibility of Kobayashi geodesics in convex domains and related properties}

\author[F. Bracci]{Filippo Bracci$^{\dag}$}
\author[N. Nikolov]{Nikolai Nikolov$^{\dag\dag}$}
\author[P. J. Thomas]{Pascal J. Thomas}

\address{F. Bracci: Dipartimento Di Matematica\\
Universit\`{a} di Roma \textquotedblleft Tor Vergata\textquotedblright\ \\
Via Della Ricerca Scientifica 1, 00133 \\
Roma, Italy}
\email{fbracci@mat.uniroma2.it}

\address{N. Nikolov:
Institute of Mathematics and Informatics\\
Bulgarian Academy of Sciences\\
Acad. G. Bonchev Str., Block 8\\
1113 Sofia, Bulgaria\\ \smallskip\newline
Faculty of Information Sciences\\
State University of Library Studies
and Information Technologies\\
69A, Shipchenski prohod Str.\\
1574 Sofia, Bulgaria}
\email{nik@math.bas.bg}

\address{P.J.~Thomas:
Institut de Math\'ematiques de Toulouse; UMR5219 \\
Universit\'e de Toulouse; CNRS \\
UPS, F-31062 Toulouse Cedex 9, France} \email{pascal.thomas@math.univ-toulouse.fr}

\begin{abstract}
Let $D\subset \C^n$ be a bounded  domain. A pair  of distinct boundary points $\{p,q\}$ of $D$ has the
\emph{visibility property} provided there exist a compact subset $K_{p,q}\subset D$ and open neighborhoods $U_p$ of $p$ and $U_q$ of $q$, such that the real geodesics for the Kobayashi metric of $D$ which join points in $U_p$ and $U_q$ intersect $K_{p,q}$. Every Gromov hyperbolic convex domain enjoys the visibility property for any couple of boundary points. The Goldilocks domains introduced by Bharali and Zimmer and the log-type domains of Liu and Wang also enjoy the visibility property.

In this paper we relate  the growth of the Kobayashi distance near
the boundary  with  visibility and provide new families of convex domains where  that property holds.
We  use the same methods to provide refinements of localization results for the Kobayashi distance,
and give a localized sufficient condition for visibility.
We also exploit visibility to study the boundary behavior of biholomorphic maps.
\end{abstract}

\thanks{$^\dag$ Partially supported by PRIN 2017 Real and Complex Manifolds: Topology, Geometry and holomorphic
dynamics, Ref: 2017JZ2SW5, by GNSAGA of INdAM and by the MIUR Excellence Department Project awarded
to the Department of Mathematics, University of Rome Tor Vergata, CUP E83C18000100006.}
\thanks{$^{\dag\dag}$ Partially supported by the National Science Fund,
Bulgaria under contract DN 12/2.}



\maketitle
\tableofcontents

\section{Introduction}
\label{intro}

A  bounded domain in $D\subset \C^n$  has the \emph{visibility property} if
the real geodesics for the Kobayashi distance  $k_D$ ``bend inside'' when connecting
points close to the (Euclidean) boundary $\partial D$ (precise definitions are given below).

Gromov hyperbolic geodesic spaces have the visibility property when considering the Gromov boundary in the Gromov topology. Therefore, in case $(D,k_D)$ is Gromov hyperbolic and the Euclidean boundary $\partial D$ is homeomorphic to the Gromov boundary, $D$ enjoys the visibility property as defined above. In fact, we show that if $(D,k_D)$ is Gromov hyperbolic, then it has the visibility property if and only if the identity map extends as a continuous surjective map from the Gromov closure of $D$ to $\overline{D}$ (see Theorem~\ref{vis-grom}).

Therefore, in light of \cite{BB}, $\mathcal C^2$-smooth bounded strongly pseudoconvex domains enjoy the visibility property. While, by \cite{BGZ}, the Euclidean end compactification of a convex domain which is Gromov hyperbolic with respect to the Kobayashi distance is naturally homeomorphic to the Gromov compactification.
Thus, the visibility property holds for Gromov hyperbolic bounded convex domains.  It is known by A. Zimmer \cite{Z1} that bounded smooth convex domains are Gromov hyperbolic with respect to the Kobayashi distance if and only if they are of finite D'Angelo type, so that finite type smooth bounded convex domains enjoy the visibility property.

In a recent paper, M. Fiacchi \cite{Fia} showed that in $\C^2$, bounded smooth pseudoconvex  domains
of finite D'Angelo type are Gromov hyperbolic and the Gromov boundary is homeomorphic to the Euclidean boundary, hence, even for those domains the visibility property holds.

However, there exist convex domains for which the visibility property holds but they are not Gromov hyperbolic (see \cite{BZ}).
We provide a class of domains enjoying the visibility property without involving any uniform quantitative
assumption about the contact between
$\partial D$ and its complex tangent plane, in particular $D$ not being Goldilocks
(see \cite[Definition 1.1]{BZ}
or Definition~\ref{gold} below) nor satisfying the weaker hypotheses of  \cite[Theorem 1.5]{BM}.

\begin{thm}
\label{onept}
Let $D \subset \C^n$ be a
bounded convex domain  with  $\mathcal C^\infty$ boundary. If
all  but finitely many  points $p \in \partial
D$ are of finite D'Angelo type, then $D$ has the visibility property.
\end{thm}

Visibility condition, although weaker than Gromov hyperbolicity, allows the study of the boundary behavior of (holomorphic) isometries (see, {\sl e.g.},  \cite{BG, BZ, CHL, Kar, Me}). In Section~\ref{visible-Gromov} we explain the underlying philosophy in the case of Gromov hyperbolic domains, showing that any biholomorphism from a bounded domain $D_1$ to a bounded domain $D_2$ which has the visibility property extends continuously to the boundary provided, for instance, $D_1$ is either strongly pseudoconvex with $\mathcal C^2$ boundary, or smooth finite type and convex or pseudoconvex and of finite type in $\C^2$. Next, based on such arguments, we localize the result (getting rid of Gromov's hyperbolicity condition) and prove the following:

\begin{thm}\label{Thm:boundary-intro}
Let $D$ and $D'$ be bounded, complete hyperbolic domains, and assume  that $\overline{D}$ has a Stein neighborhood basis and $D'$ has the visibility property. Suppose there exists $p\in\partial D$ such that $\partial D$ is $\mathcal C^3$-smooth and strongly pseudoconvex at $p$. If $F$  is a biholomorphism from $D$ to $D'$ then  $F$ admits a non-tangential limit at $p$.
\end{thm}

As it follows from the proof, the same result holds provided $D$ is any complete hyperbolic domain such that it has a geodesic ray $\gamma$ landing at a point $p\in \partial D$, where $\partial D$ is $\mathcal C^1$-smooth, so that any non-tangential sequence in $D$ converging to $p$ stays at finite hyperbolic distance from $\gamma$.

To find sufficient conditions for the visibility property, it is possible to forego the global hypothesis of convexity, and find
analogues with suitably localized hypotheses. A point $p$ on the boundary of a domain $D$ is a {\sl locally $\C$-strictly convex point} if there exists  an open neighborhood $U$ of $p$ and a biholomorphism $\Psi:U\to \Psi(U)$ so that $\Psi(U\cap D)$ is convex and every complex affine line  $L$ which contains $\Psi(p)$ and verifies
$L\cap \Psi(U\cap D)=\emptyset$ has the property that $L\cap \Psi(\overline{U\cap D})=\{p\}$ (see Definition \ref{cvxble}). In Section~\ref{localization} we prove Theorem~\ref{se}, which has the following consequence:

\begin{thm}
\label{Thm:se-intro}
Let $D$ be a complete hyperbolic bounded domain with Dini-smooth boundary and assume all boundary points are locally $\C$-strictly convex. Then $D$ has the visibility property.
\end{thm}

Local visibility implies
localization properties of the Kobayashi distance, which substantially generalize \cite[Theorem 1.4 and Theorem 3.9]{LW}.

\begin{thm}
\label{visloc}
Let $D$ be a domain in $\C^n$, $p\in \partial D$, and $U$ a neighborhood of $p$ such that $D\cap U$
is convex and
enjoys the visibility property. Then for any neighborhood $V$ of $p$ such that
$V\subset \subset U$, there exists $C>0$ such that for any $z,w \in V$,
\[
k_{D \cap U}(z,w) \le k_D(z,w) + C.
\]
\end{thm}

The plan of the paper is the following. In Section~\ref{Sec:prel} we introduce some preliminary notations and results. In Section~\ref{visible-Gromov} we consider Gromov hyperbolic domains and visible domains and prove Theorem~\ref{Thm:boundary-intro}. In Section~\ref{visgro} we recall the connection between visibility and the Gromov product, relate it to
precise growth of the Kobayashi distance on convex domains, and extend some results of \cite{Z3} to
slightly larger classes of domains, obtaining Corollary \ref{trivface}. In Section~\ref{infinite-type}, we prove Theorem~\ref{onept} and give a family of examples of smooth bounded convex domains with (non isolated) boundary points of infinite type which do (and do not) satisfy the visibility property. Finally, in Section~\ref{localization}, we investigate how  to ``localize'' the previous results replacing the convexity assumption with local conditions near the boundary and prove Theorem~\ref{Thm:se-intro}
and Theorem~\ref{visloc}.

We would like to thank Vikramjeet Singh Chandel, Anwoy Maitra, and Amar Deep Sarkar, who pointed out a gap in
the proof of Lemma \ref{sameheight} in a previous version of the manuscript.

\section{Definitions and Preliminaries}\label{Sec:prel}
In this section, $D\subset \C^n$ is a domain (connected open set). The domain is called \emph{$\C$-convex} if any non-empty
intersection with a complex line is  simply connected. We say $D$ is of \emph{finite type} if the order of contact
of any analytic disc in $\C^n \setminus D$ with $D$ is finite.  For convex domains, it is enough to consider
the order of contact of complex lines \cite{McN}, and this holds also for $\C$-convex domains \cite[Proposition 6]{NPZ}.

\begin{defn}
\label{defkob}
Let $x, y, z\in D$ and $X\in\C^n.$ The Kobayashi-Royden (pseudo)metric $\k_D$
and the Kobayashi (pseudo)distance $k_D$ of $D$ are defined as:
\begin{equation*}
\label{defgeod}
\k_D(z;X)=\inf\{|\alpha|:\exists\varphi\in\mathcal O(\D,D):
\varphi(0)=z,\alpha\varphi'(0)=X\},
\end{equation*}
\begin{equation}
\label{defkoba}
k_D(x,y)=\inf_\gamma\int_0^1\k_D(\gamma(t);\gamma'(t))dt,
\end{equation}
where the infimum is taken over all piecewise $\mathcal C^1$ curves $\gamma:[0,1]\to D$ with $\gamma(0)=x$
and $\gamma(1)=y.$

A \emph{geodesic} for $k_D$ is a curve
$\sigma: I \longrightarrow D$, where $I$ is an interval in $\R$, such that
for any $ s,t \in I$, $k_D ( \sigma (s), \sigma (t) ) = |s-t|$. Moreover:
\begin{itemize}
\item[-] If  $x,y \in D$,  $I=[0,L]$ and  $\sigma(0)=x, \sigma(L)=y$, then $L=k_D(x,y)$ and
we say that $\s$ is a geodesic joining $x$ and $y$.
\item[-] If $I=(-\infty,+\infty)$, we say that $\s$ is a \emph{geodesic line}.
\item[-] If $I=[0,\infty)$, we say that $\s$ is a \emph{geodesic ray}. \end{itemize}
A geodesic ray $\sigma$ \emph{lands} if there exists $p\in \overline{D}$ such that $\lim_{t\to\infty}\sigma(t)=p$.
\end{defn}

The domain $D$ is {\sl complete hyperbolic} if it is Kobayashi hyperbolic and if $k_D$ is a complete distance, or equivalently, the balls for the Kobayashi distance are relatively compact. Bounded convex domains are well known to be complete hyperbolic.

If a domain $D$ is complete hyperbolic, by the Hopf-Rinow Theorem, $(D, k_D)$ is a  geodesic space,  namely,  any two points in $D$ can be joined by a geodesic.

It is also a fact that when a geodesic exists, it realizes the minimum in \eqref{defkoba}
\cite[Theorem 3.1]{Ven}.

\begin{defn}
\label{defvis}
Let $D$ be a complete hyperbolic domain in $\C^n$. Let $p,q \in \partial D$, $p\neq q$. We say that the pair $\{p,q\}$ has \emph{visible
geodesics}
if there exist neighborhoods $U, V$ of $p, q$ respectively such that $\overline U \cap \overline V = \emptyset$,
and a compact set $K \subset D$ such that for any geodesic $\gamma:[0,L]\rightarrow D$
with $\gamma(0)\in U$, $\gamma(1)\in V$, then $\gamma ([0,L]) \cap K \neq \emptyset$.

We say that $D$ has the \emph{visibility property} if any pair $\{p,q\} \subset \partial D$, $p\neq q$,
has visible  geodesics.
\end{defn}

Note that   the pair $\{p,q\}$ has visible geodesics
if and only if there exists a  compact set $K \subset D$ such that
for any sequences $(p_k)_k, (q_k)_k \subset D$, with $p_k\to p$, $q_k\to q$,
then for $k$ large enough any geodesic joining $p_k$ to $q_k$ intersects $K$.

Our definition of visibility  requires that $D$ is complete hyperbolic in order for the condition of existence of geodesics joining any two points of $D$ to be non vacuous. In  \cite{BZ, BM},  the authors  defined the notion of visibility for not necessarily complete hyperbolic domains by using ``almost-geodesics'', proving that any two points of a bounded domain can be joined by a $(1,\epsilon)$-almost-geodesic. In fact, there is another reason for considering almost-geodesics instead of just geodesics even in complete hyperbolic domains: in general, almost-geodesics are easier to find than geodesics. However, in this paper we mainly consider convex domains, which are complete hyperbolic and, for the sake of simplicity, we decided to deal only with geodesics. The interested reader can check that the following arguments work as well considering almost-geodesics instead of geodesics and the general notion of visibility as in \cite{BZ, BM}.

 We now recall a notion related to Gromov hyperbolicity.

\begin{defn}
\label{groprod}
Let $D$ be a domain. Choose a base point $o\in D$. Let $x,y\in D$. The \emph{Gromov product} of $x$ and $y$ with respect to $o$ is
\begin{equation}
(x|y)_o :=\frac{1}{2} [k_D(x,o)+k_D(o,y)-k_D(x,y)].
\end{equation}
\end{defn}

One is usually interested in understanding whether such a Gromov product converges to infinity or not when $x, y$ tend to the boundary. For this aim, the choice of the base point $o$ is completely irrelevant since  $\left|(x|y)_o-(x|y)_{o'}\right| \le k_D(o,o')$.

Visibility  implies that the Gromov product of sequences tending to different boundary points is bounded from above.

\begin{prop}
\label{necvis}
Let $D$ be a complete hyperbolic  domain in $\C^n$, and
$p,q \in \partial D$. If the pair $\{p,q\}$ has visible  geodesics then
for any base point $o \in D$,
\begin{equation}
\label{stdest}
\limsup_{(x, y)\to (p,q)} (x|y)_o < \infty.
\end{equation}
\end{prop}

\begin{proof}
By hypothesis, there exists a compact set $K$ such that, whenever $x, y\in D$ are so that $x$ is sufficiently close to $p$ and $y$ to $q$, every geodesic  $\gamma$ joining $x$ and $y$ intersects $K$.

In particular, for any such $x,y$ and geodesic $\gamma$, there is a point $z\in K$ belonging to $\gamma$. Hence,
\begin{equation*}
\begin{split}
k_D(x,y)&=k_D(x,z)+k_D(z,y) \ge k_D(x,o)+k_D(o,y)-2k_D(o,z)
\\
&\ge k_D(x,o)+k_D(o,y)-2 \max_{\zeta\in K} k_D(o,\zeta),
\end{split}
\end{equation*}
and we are done.
\end{proof}

The converse of the previous result holds under global assumptions (compare with \cite[Proposition 5.1(2)]{Mai}):
\begin{prop}
\label{suffvis}
Let $D$ be a bounded complete hyperbolic domain in $\C^n$.
Then $D$ has the visibility property
if and only if $\limsup_{(x, y)\to (p,q)} (x|y)_o < \infty$ for any  $p, q  \in \partial D$, $p\neq q$.
\end{prop}

\begin{proof}
One implication follows from Proposition \ref{necvis}.

The converse follows the lines of the proof of \cite[Theorem 6.1]{Z3}.
Suppose the pair $\{p,q\}$ does not have visible geodesics.
So there exist a sequence $\{x_k\}$ converging to $p$, a sequence $\{y_k\}$ converging to $q$ and a sequence of geodesics $\gamma_k$ joining $x_k$ to $y_k$ so that the image of $\gamma_k$ eventually avoids  every compactum in $D$.

Let $z_k$ in $\gamma_k$ be such that $\|x_k-z_k\|=\|z_k-y_k\|$. By compactness of $\overline D$, we can assume, up to  subsequences, that $\{z_k\}$ converges
to a point $p_1 \in \overline D$ such that $\|p-p_1\|=\|p_1-q\|$.  Since $\gamma_k$ eventually avoids any compactum in $D$, it follows that $p_1 \notin D$, hence $p_1 \in \partial D$. Therefore,  since by hypothesis there exists $C>0$ such that $2(x_k|z_k)_o +2(y_k|z_k)_o<C$ for all $k$, we have
\begin{equation*}
\begin{split}
k_D(x_k,o)+k_D(y_k,o)& \ge k_D(x_k,y_k) =k_D(x_k,z_k)+k_D(z_k,y_k)\\& =   - 2(x_k|z_k)_o-2(y_k|z_k)_o+k_D(x_k,o)+k_D(y_k, o)+2k_D(z_k,o)\\&\geq k_D(x_k,o) + 2 k_D(z_k,o) + k_D(y_k,o) - C.
\end{split}
\end{equation*}
Therefore, $k_D(z_k,o)\leq C$, which, by the completeness of $(D, k_D)$, implies that $p_1\in D$, a contradiction.
\end{proof}

As customary, for a point $z\in D$, we let
\[
\d_D(z):=\inf\{\|z-w\| : w\in \C^n\setminus D\},
\]
be the (Euclidean) distance of $z$ from the boundary of $D$.

Since the Kobayashi distance is always larger than the Carath\'eodory distance,
the estimate in \cite[(2)]{N2} (coming from the proof of \cite[Theorem 5.4]{Blo}) yields that for $x,y\in D$,
\begin{equation}
\label{dumbest}
k_D(x,y) \ge \frac12 \left| \log \frac{\d_D(y)}{\d_D(x)}\right|.
\end{equation}
The previous estimate is pretty good when $x$ goes to the boundary and $y$ stays compactly in $D$. However, when both $x, y$ go to the boundary, it does not give much information. In fact,
in many cases  when $x$ and $y$ converge to different boundary points, the Kobayashi distance between the two explodes  (this is the case for instance if $D$ is strongly (pseudo)convex with $\mathcal C^2$ boundary, see, {\sl e.g.}, \cite[Corollary 2.3.55]{Aba}).

In a convex domain, by \eqref{dumbest},  condition \eqref{stdest} implies the following \emph{``log-estimate''}
\begin{equation}
\label{logest}
\limsup_{(x, y)\to (p,q)} \left(\frac12 \log\frac1{\d_D(x)} + \frac12 \log\frac1{\d_D(y)} -k_D(x,y) \right)< \infty.
\end{equation}

We fix some terminology to describe upper estimates that are counterparts to \eqref{dumbest}.

\begin{defn}
\label{growth}
We say that:
\begin{enumerate}
\item
$D$ has \emph{$\a$-growth} if there exist  some $\a>0$, $\beta>0$
and $x_0\in D$ such that
$$
\sup_{z\in D} \left( k_D(x_0,z)- \frac{\beta}{ \d_D(z)^{\a}} \right) <\infty.
$$
\item
$D$ has \emph{$\a$-log-growth} if there exist  some  $\a>0$
and $x_0\in D$ such that
$$
\sup_{z\in D} \left( k_D(x_0,z)- \a\log \frac{1}{\d_D(z)} \right) <\infty.
$$
\end{enumerate}
\end{defn}

Note that in particular when $\partial D$ is
Lipschitz, $D$ has $\a$-log-growth \cite[Lemma 2.3]{BZ}. In particular convex domains  automatically satisfy
that property.

Notice that if $D$ has $\frac{1}{2}$-log-growth, then when the log estimate \eqref{logest} holds,
the Gromov product $(x|y)_o$ remains bounded, i.e. \eqref{stdest} holds.

Moreover, recall that a point $p\in \partial D$ is a {\sl Dini-smooth point} if $\partial D$ is $\mathcal C^1$-smooth at $p$ and the inner unit normal vector $\nu_q$  is a Dini-continuous function for $q$ close to $p$, namely,
there is a neighborhood $U$ of $p$ such that
\[
\int_0^1 \frac{\omega(t)}{t}dt<+\infty,
\]
where $\omega(t):=\sup\{\|\nu_x-\nu_y\|: \|x-y\|<t, x,y\in U\cap\partial D\}$. The boundary $\partial D$ is {\sl Dini-smooth} if all its  points are Dini-smooth points.

If $\partial D$ is Dini-smooth, then $D$ has $\frac{1}{2}$-log-growth
by \cite[Corollary 8]{NA},  or \cite[Proposition 1.5]{Mai} (see also \cite[Proposition 4.3]{Z3}
in the convex case). \cite[Proposition 4.6]{Mai}
gives the same conclusion with slightly relaxed hypotheses.
However, there exist $\mathcal C^1$-smooth (but not Dini-smooth) domains for which $D$ does not
have $\frac12$-log-growth \cite[Example~2]{NPT}.

In particular, in Dini-smooth bounded convex domains,
condition \eqref{stdest} is equivalent to the log-estimate \eqref{logest}.

Finally, we recall a definition from \cite{BZ}.
\begin{defn} \cite[Definition 1.1]{BZ}
\label{gold}
Let
\[
M_D(r) := \sup\left\{ \frac{1}{\k_{D}(x;X)} : \d_D(x) \leq r, \|X\|=1\right\}.
\]
A bounded domain $D \subset \C^n$ is a \emph{Goldilocks domain} if
\begin{enumerate}
\item for some $\var >0$ we have
\begin{align*}
\int_0^\var \frac{M_D\left(r\right)}{r}dr < \infty,
\end{align*}
\item $D$ has $\a$-log-growth for some $\a>0$.
\end{enumerate}
\end{defn}

The first property is satisfied, for example,
for any bounded pseudoconvex domain of finite type (in sense of D'Angelo)  \cite[Lemma 2.6]{BZ},
but the class of domains with the Goldilocks property strictly contains the class of domains of finite type.

It is proved in \cite[Theorem 1.4]{BZ} that a Goldilocks domain has an extended visibility property
which implies in particular the one in Definition \ref{defvis}  when the Kobayashi distance is geodesic.
A more general result implying visibility is given in \cite[Theorem 1.5 (General Visibility Lemma)]{BM}, which essentially reduces to the
Goldilocks case when the domain is convex.

\begin{defn}
A \emph{complex face} of a convex domain $D$ is  $L\cap \partial D$, where
$L$ is an affine complex line such that $L\cap D=\emptyset$ and $L\cap \partial D\neq \emptyset$.

Given a point $p\in \partial D$, the \emph{multiface} $F_p$ of $p$ is the union of all the
complex faces of $D$ which contain $p$.
\end{defn}

Let $L$ be an affine complex line such that $L\cap D=\emptyset$ and $p\in L\cap\partial D$. Since $L$ and $D$ are convex and disjoint and $D$ open, by Hahn-Banach's extension theorem, there exists a vector $v\in \C^n$ such that $\Re \langle z-p, v\rangle<0$ for all $z\in D$ and $\Re \langle z-p, v\rangle\geq 0$ for all $z\in L$ (here $\langle \cdot, \cdot \rangle$ denotes the usual Hermitian product in $\C^n$). Actually, since $L$ is a complex affine line it turns out that $\langle z-p, v\rangle=  0$ for all $z\in L$. This implies that $L$ is contained in the complex affine hyperplane $H:=\{z\in \C^n: \langle z-p, v\rangle=  0\}$. Such a hyperplane $H$ is called a {\sl complex supporting hyperplane} of $D$ at $p$ (see \cite[Def. 3]{AR}). The intersection of $\overline{D}$ with all complex supporting hyperplanes of $D$ at $p$ is denoted by $\hbox{Ch}(p)$. In \cite{AR}, the point $p$ is called a {\sl strictly $\C$-linearly convex point} if $\hbox{Ch}(p)=\{p\}$. By the previous remark, a point for which the multiface $F_p=\{p\}$ is a strictly $\C$-linearly convex point. The converse is not true in general: the multiface of the bidisc $\D\times \D$ at $(1,1)$ is $(\overline{\D}\times \{1\})\cup (\{1\}\times \overline{\D})$, while $\hbox{Ch}((1,1))=\{(1,1)\}$.

\section{Gromov hyperbolic domains, visibility and the proof of Theorem~\ref{Thm:boundary-intro}}
\label{visible-Gromov}

The aim of this section is to see how visibility is related to Gromov visibility in a bounded domain for which the Kobayashi distance is complete and Gromov hyperbolic.

We briefly recall the definition of Gromov compactification for a complete hyperbolic domain $D$ such that $(D, k_D)$ is Gromov hyperbolic. Let $z_0\in D$. Let $\Gamma_{z_0}$ be the set of all geodesic rays $\gamma$ such that $\gamma(0)=z_0$. Two geodesic rays $\gamma, \eta\in \Gamma_{z_0}$ are equivalent if
\[
\sup_{t\geq 0}k_D(\gamma(t), \eta(t))<+\infty.
\]
The set of equivalence classes in $\Gamma_{z_0}$ is the Gromov boundary $\partial_G D$. Let $\overline{D}^G:=D\cup \partial_G D$ . One can give $\overline{D}^G$ a topology which makes it a first countable, Hausdorff, compactification of $D$. In  this topology, a sequence $\{\sigma_k\}\subset \partial_GD$ converges to $\sigma\in \partial_GD$ if for every representative $\gamma_k\in \sigma_k$, every subsequence of $\{\gamma_k\}$ has a subsequence which converges uniformly on compacta to a geodesic ray $\gamma\in \sigma$. Also, a sequence $\{z_k\}\subset D$ converges to $\sigma\in \partial_G D$ if, for every geodesic $\gamma_k:[0,R_k]\to D$ such that $\gamma_k(0)=z_0$ and $\gamma_k(R_k)=z_k$, then every subsequence of $\{\gamma_k\}$ has a subsequence that converges uniformly on compacta to a geodesic ray $\gamma\in \sigma$.

%We start with the following simple result:

%\begin{lemma}
%\label{bddcv}
%Let $D$ be a bounded, complete hyperbolic domain with the visibility property.
%If $\{z_k\}, \{w_k\}\subset D$ converge to different points on the boundary of $D$, then $k_D(z_k,w_k)\to\infty$.
%\end{lemma}
%\begin{proof}
%Consider a geodesic $\gamma_k$ joining $z_k$ to $w_k$.  By the visibility hypothesis,
%there exist a compact set $K\subset D$ and $t_k$ such that $\gamma_k(t_k) \in K$ for all $k$.
%Since $k_D$ is complete, this implies that $k_D(z_k, \gamma_k(t_k)) $ and $k_D(w_k, \gamma_k(t_k)) \to\infty$,
%thus since $\gamma_k$ is a geodesic, $k_D(z_k, w_k)= k_D(z_k, \gamma_k(t_k) )+ k_D(w_k, \gamma_k(t_k)) \to\infty$.
%\end{proof}
First notice that as a consequence of Proposition \ref{necvis}, if $D$ is a bounded, complete hyperbolic
domain with the visibility property, and $\{z_k\}, \{w_k\}\subset D$ converge to different points on the boundary
of $D$, then $k_D(z_k,w_k)\to\infty$.

\begin{lemma}
\label{raypt}
Let $D$ be a bounded, complete hyperbolic domain with the visibility property.
Then any geodesic ray lands at a  point  on  $\partial D$.

Conversely, let $z_0\in D$. Let $\{z_k\}\subset D$ be a sequence which converges to a point $p\in \partial D$, and let  $\gamma_k$ be geodesics joining $z_0$ to $z_k$. Then, up to subsequences, $\{\gamma_k\}$ converges uniformly on compacta to a  geodesic ray landing at $p$.
\end{lemma}
\begin{proof}
Let $\gamma$ be a geodesic ray in $D$. Since $D$ is bounded and complete hyperbolic, the cluster set $\Gamma$ of $\gamma$ at $+\infty$ is contained in $\partial D$.

 Suppose there are two
distinct  points $p, q\in \Gamma$, and sequences $s_k\to\infty$ and $t_k\to\infty$ such that
$\lim_{k\to\infty} \gamma(s_k)=p$, $\lim_{k\to\infty} \gamma(t_k)=q$. Passing to subsequences if necessary, we may assume $s_k < t_k < s_{k+1}$ for all $k$. Then by the visibility property
applied to the geodesics $\gamma|_{[s_k,t_k]}$, there is a
compact set $K\subset D$ such that for each $k$, there exists $s'_k\in (s_k,t_k)$ such that we have $\gamma(s'_k)\in K$.
But then  $\Gamma\cap K\neq \emptyset$, a contradiction.

Conversely, fix $z_0 \in D$ and a sequence $\{z_k\}$ converging to $p$. Let $\gamma_k$ be a geodesic joining $z_0$ to $z_k$.
By Arzel\`a-Ascoli's theorem, and a  diagonal argument, up to subsequences, we can assume that $\{\gamma_k\}$ is uniformly convergent on compacta  to a geodesic ray $\gamma$. From what we already proved $\gamma$ lands at some point  $q\in \partial D$. We have to show that $q=p$. If this is not the case, however,  we can find sequences of positive real numbers $\{s_k\}$ and $\{t_k\}$ converging to $+\infty$, and we can assume $s_k<t_k$, such that $\gamma_k(s_k)\to q$ and $\gamma_k(t_k)\to p$. But then, by the visibility property, $\gamma_k|_{[s_k, t_k]}$ has to intersect a given compact set $K\subset D$ for all $k$, say, $\gamma_k(t_k')\in K$ for some $t_k'\in [s_k, t_k]$. But,
\[
t_k'= k_D(z_0, \gamma_k(t_k'))\leq \max_{z\in K}k_D(z_0, z)<+\infty,
\]
a contradiction. Hence $p=q$ and we are done.
\end{proof}

The visibility property in Gromov hyperbolic spaces has a strong consequence. As a matter of notation, let $X$ be a compactification of a domain $D$. We say that a geodesic line $\gamma:(-\infty, +\infty)\to D$ is a {\sl geodesic loop in $X$} if $\gamma$ has the same cluster set in $X$ at $+\infty$ and $-\infty$. We say that {\sl $D$ has no geodesic loops in $X$} provided there is no geodesic line in $D$ which is a geodesic loop in $X$.

It is easy to see that, if $(D, k_D)$ is Gromov hyperbolic and $\gamma$ is a geodesic line, and if we let $\sigma^{\pm}:=\lim_{t\to \pm\infty}\gamma(t)\in \partial_G D$ (the limit understood in the Gromov topology), then $\sigma^+\neq \sigma^-$. Thus $D$   has no geodesic loops in $\overline{D}^G$.

\begin{rem}\label{Rem:refereeZ}
The notion of (not having) geodesic loops in some compactification is strictly related to the notion of {\sl good compactification} introduced in \cite[Def. 6.2]{BZ}. In fact, it is easy to see that if $D^\ast$ is a good compactification in the sense of \cite{BZ}, then $D$ has no geodesic loops in $D^\ast$. Conversely, using Lemma~\ref{raypt} it is easy to see that, if $D$ is a bounded, complete hyperbolic domain with the visibility property, then $D$ has no geodesic loops in $\overline{D}$ if and only if $\overline{D}$ is a good compactification of $D$ in the sense of \cite{BZ}.
\end{rem}

\begin{thm}\label{vis-grom}
Let $D$ be a complete hyperbolic bounded domain. Assume $(D, k_D)$ is Gromov hyperbolic. Then  $D$ has the visibility property if and only if the identity map extends as a continuous surjective map $\Phi:\overline{D}^G\to\overline{D}$. Moreover, $\Phi$ is a homeomorphism if and only if $D$ has no geodesic loops in $\overline{D}$.
\end{thm}
\begin{proof}
If the identity map extends as a surjective continuous map from $\overline{D}^G$ to $\overline{D}$, since $(D,k_D)$ is Gromov hyperbolic and hence $\overline{D}^G$ enjoys the Gromov visibility property, it follows at once that $D$ has the visibility property. Moreover, if the extension is a homeomorphism, then $D$ has no geodesic loops because $\overline{D}^G$ has none.

Conversely, assume that $D$ has the visibility property. We first define a map $\Phi:\overline{D}^G\to \overline{D}$ as follows: if $z\in D$ then $\Phi(z)=z$. If $\sigma\in \partial_G D$, let $\gamma\in \sigma$ be such that $\gamma\in \Gamma_{z_0}$. Then by Lemma~\ref{raypt}, $\gamma$ lands at some point $p\in \partial D$. Note that by Proposition \ref{necvis} %Lemma~\ref{bddcv},
the point $p$ does not depend on the representative $\gamma$ chosen. Therefore, we can set $\Phi(\sigma):=p$.

By Lemma~\ref{raypt}, the map $\Phi$ is surjective and $\Phi(z_k)\to \Phi(\sigma)$ if $\{z_k\}\subset D$ is a sequence which converges in the Gromov topology to $\sigma\in \partial_G D$. An argument similar to the one in the proof of Lemma~\ref{raypt} shows that $\Phi(\sigma_k)\to \Phi(\sigma)$ if $\{\sigma_k\}\subset \partial_G D$ converges to $\sigma$ in the Gromov topology.

Assume now that $D$ has no geodesic loops. Since $\overline{D}^G, \overline{D}$  are Hausdorff and compact, if we show that $\Phi$ is injective, then it is also a homeomorphism. Assume by contradiction that $\Phi(\sigma)=\Phi(\theta)$ for some $\sigma,\theta\in \partial_G D$, $\sigma\neq \theta$.  Let $\gamma\in \sigma$ and $\eta\in\theta$ be two representatives. Hence, there exists a sequence $\{t_k\}$ converging to $+\infty$ such that
\[
\lim_{k\to \infty}k_D(\gamma(t_k), \eta(t_k))=+\infty.
\]
We first claim that
\[
\lim_{k\to \infty}\inf_{s\geq 0}k_D(\eta(s), \gamma(t_k))=+\infty.
\]
Indeed, assume by contradiction that this is not the case. Then there exists a sequence $\{s_k\}$ of positive numbers and $R>0$ such that for all $k$
\[
k_D(\eta(s_k), \gamma(t_k))=\min_{s\geq 0}k_D(\eta(s), \gamma(t_k))\leq R.
\]
Since
\[
|t_k-s_k|=|k_D(z_0, \gamma(t_k))-k_D(z_0, \eta (s_k))|\leq k_D(\eta(s_k), \gamma(t_k))\leq R,
\]
we have for all $k$,
\[
k_D(\eta(t_k), \gamma(t_k))\leq k_D(\eta(s_k), \gamma(t_k))+k_D(\eta(s_k), \eta(t_k))\leq 2R,
\]
a contradiction. Therefore the claim holds.

Now, let $\beta_k:[0, T_k]\to D$ be a geodesic joining $\gamma(t_k)$ and $\eta(t_k)$. Since $(D, k_D)$ is Gromov hyperbolic, it follows that there exists $M>0$ such that for all $k$ and all $m\leq k$,
\[
\min\{k_D(\gamma(t_m), \eta([0,t_k])), k_D(\gamma(t_m), \beta_k([0,T_k]))\}\leq M.
\]
By the claim, if $m$ is sufficiently large, the only possibility is that $k_D(\gamma(t_m), \beta_k([0,T_k]))\leq M$. This implies that there exists a compact set $K\subset D$ such that $\beta_k$ intersects $K$ for all $k$. We can reparametrize $\beta_k$ in such a way that $\beta_k:[-s_k, r_k]\to D$ and $\beta_k(0)\in K$. Using such a parametrization and arguing as in the proof of Lemma~\ref{raypt}, we see that by visibility, $\beta_k$ converges to a geodesic line $\beta:(-\infty, +\infty)\to D$ and that $\lim_{t\to \pm \infty}\beta(t)=p$. Namely, $\beta$ is a geodesic loop, contradicting our hypothesis.
\end{proof}

In dimension one, we have a simple characterization of visible simply connected domains:

\begin{cor}
Let $D\subset \C$ be a bounded simply connected domain. Then $D$ has the visibility property if and only if $\partial D$ is locally connected.
\end{cor}
\begin{proof}
Assume $\partial D$ is locally connected. Let $f:\D \to D$ be a Riemann map. Since $(\D, k_{\D})$ is Gromov hyperbolic, so is $(D, k_D)$. Since $\partial D$ is locally connected, it follows by the Carath\'eodory extension theorem (see, {\sl e.g.} \cite[Thm. 4.3.1]{BCD}) that  $f$ extends continuously to the boundary, call $\tilde f$ such an extension. Now, $f^{-1}$ extends as a homeomorphism--denote it by $\hat{f}^{-1}$---from $\overline{D}^G$ to $\overline{\D}^G$   and the identity map ${\sf id}_{\D}$ extends as a homeomorphism $\hat{\sf id}_{\D}$ from  $\overline{\D}^G$ to $\overline{\D}$. Therefore, $\hat{\sf id}_{D}:= \tilde f\circ \hat{\sf id}_{\D}\circ  \hat{f}^{-1}$ is a  continuous surjective extension of ${\sf id}_{D}$. Thus, by Theorem~\ref{vis-grom}, $D$ has the visibility property.

Conversely, if $D$ has the visibility property, then by Theorem~\ref{vis-grom}, ${\sf id}_D$ extends as a surjective continuous map from $\overline{D}^G$ to $\overline{D}$. Since $\partial_G D$ is homeomorphic to $\partial_G \D$ and the latter is homeomorphic to $\partial \D$, it follows that there exists a continuous surjective function from $\partial \D$ to $\partial D$, and hence $\partial D$ is locally connected (see, {\sl e.g.} \cite[Thm. 4.3.1]{BCD}).
\end{proof}

The following example provides a domain with the visibility property and geodesic loops:

\begin{exam}
The domain $D:=\D\setminus\{[0,1)\}$ is simply connected and $\partial D$ is locally connected, thus, $D$ has the visibility property. However, $D$ has geodesic loops: let $f:\D \to D$ be a Riemann map and take $\xi_1, \xi_2\in \partial \D$, $\xi_1\neq \xi_2$, such that $f(\xi_1)=f(\xi_2)=1/2$ (see, {\sl e.g.} \cite[Prop. 4.3.5]{BCD}), let $\gamma$ be the geodesic line in $\D$ whose closure contains $\xi_1$ and $\xi_2$. Hence, $f(\gamma)$ is a geodesic loop in $D$.
\end{exam}

Since any  biholomorphism between two Gromov hyperbolic domains extends as a homeomorphim between the Gromov closure of the domains,  we have also the following direct consequence (compare with \cite[Theorem 6.5]{BZ} in light of Remark~\ref{Rem:refereeZ}):

\begin{prop}
Let $D_1, D_2\subset \C^n$ be complete hyperbolic bounded domains. Assume $D_1, D_2$ have the visibility property. If $(D_1, k_{D_1})$ is Gromov hyperbolic and $D_1$ has no geodesic loops then every biholomorphism $F:D_1\to D_2$ extends continuously on $\partial D_1$.
\end{prop}

In particular the previous proposition applies when $D_1$ is strongly pseudoconvex, or convex and Gromov hyperbolic (so, for instance if $\partial D$ is smooth and of finite type), or $D_1\subset \C^2$ is smooth, pseudoconvex and of finite type.

One can localize the previous argument, by getting rid of Gromov hyperbolicity condition, provided one has some knowledge on geodesic rays landing at a given boundary point. This is the content of Theorem~\ref{Thm:boundary-intro}, which we are now going to prove:

\begin{proof}[Proof of Theorem~\ref{Thm:boundary-intro}]
By \cite[Thm. 2.6]{BFW} (which is actually based on \cite[Thm 1.1]{DFW} and allows to replace the hypothesis that $D$ is strongly pseudoconvex in \cite[Thm. 2.6]{BFW} with $\overline{D}$ having a Stein neighborhood basis), there exist a $C^3$-smooth bounded strongly convex domain $W\subset \C^n$ and a univalent map $\Phi:D\to \C^n$ such that $\Phi$ extends $C^3$ up to $\overline{D}$, $\Phi(p)\in \partial W$, $\Phi(D)\subset W$ and there exists an open neighborhood $U$ of $\Phi(p)$ such that $W\cap U= \Phi(D)\cap U$.  By \cite[Lemma 4.5]{BST}, there exists a complex geodesic for $W$  ({\sl i.e.}, an isometry between $k_\D$ and $k_W$) $\varphi:\D \to W$  such that $\varphi(\D)\subset U$, $\varphi(1)=p$ and $\varphi$ is also a complex geodesic for $\Phi(D)$. Note that $[0,1)\ni r\mapsto \varphi(r)$ is (once suitably reparametrized in hyperbolic arc-length) a geodesic ray in $\Phi(D)$, which lands at $p$ and it is transverse to $\partial D$ (by Hopf's Lemma). It follows that  $D$ has a geodesic ray $\gamma$ landing at $p$ non-tangentially.

Since $F(\gamma)$ is a geodesic ray in $D'$,  by Lemma~\ref{raypt}, it lands at some point $q\in \partial D'$.

In order to prove the theorem, we need to show that if $\{z_k\}\subset D$ is a sequence converging non-tangentially to $p$ then $\{F(z_k)\}$ converges to $q$.

In order to see this, fix a sequence $\{z_k\}\subset D$ converging non-tangentially to $p$. Let $B\subset D$ be a ball tangent to $\partial D$ at $p$. Since $k_D|_B\leq k_B$, it follows from \cite[Lemma 2.3]{BF} that $\{z_k\}$  stays at finite Kobayashi distance  from $\gamma$. Hence, there exist a sequence of positive real numbers $\{t_k\}$ converging to $+\infty$ and a constant $C>0$ such that $k_D(\gamma(t_k), z_k)\leq C$ for all $k$. Thus, $k_{D'}(F(z_k), F(\gamma(t_k))\leq C$  for all $k$, and,
by Proposition \ref{necvis}
%Lemma~\ref{bddcv}
it follows that $\{F(z_k)\}$ converges to $q$.
\end{proof}

\section{Visibility and growth of the metric}
\label{visgro}

In this section, we explore the possibility of a converse to Proposition \ref{necvis}. By results of Zimmer \cite[Propositions 3.5 and 4.6]{Z2},  condition \eqref{stdest}  excludes the presence of analytic discs with nonempty interior in $\partial D$.

Our first result generalizes \cite[Lemma 4.5]{Z3} to the case of domains with irregular boundaries. In this case,  some boundary points do not have a well defined tangent hyperplane, and then the notion of multiface  is needed.

\begin{prop}
\label{suffstd}
Let $D$ be a bounded convex domain, and $p,q \in \partial D$ such that $F_p \cap F_q = \emptyset$.
Then  the log-estimate \eqref{logest} holds.

In particular,   if $D$ has $\frac12$-log-growth,
then  $\limsup_{(x, y)\to (p,q)} (x|y)_o < \infty$.
\end{prop}

\begin{proof}
Let $\{x_k\}\subset D$ be a sequence converging to $p$ and $\{y_k\}\subset D$  a sequence converging to $q$. For each $k$,
choose a point $p_k \in \partial D$ such that $\|x_k-p_k\|=\d_D(x_k)$. Then, because $B(x_k,\d_D(x_k))\subset D$
and $p_k\in \overline B(x_k,\d_D(x_k))$, there exists a unique (real) tangent hyperplane $T_{p_k}$, and therefore
a unique complex tangent hyperplane $H_{p_k}$, to $\partial D$
at $p_k$. Up to passing to subsequences, we can assume that $H_{p_k}$ converges (in the Hausdorff
distance when restricted to a fixed ball) to $H_p$, a complex supporting hyperplane for $D$ at $p$. Then $H_p\cap \overline D \subset F_p$.
In the same way, we get $H_q$, a complex supporting hyperplane for $D$ at $q$.

Since $F_p\cap F_q=\emptyset$, it follows that $(H_p\cap H_q)\cap \overline D =\emptyset$ and,  $D$ being bounded, the Euclidean distance of $H_p\cap \overline{D}$ from $H_q\cap \overline{D}$ is positive. We write, for $\eta>0$,
\begin{equation}
\label{tube}
\mathcal N_p (\eta) := \left\{ z\in \C^n: \mbox{dist}(z, H_p) \le \eta \right\},
\quad
\mathcal N_q (\eta) := \left\{ z\in \C^n: \mbox{dist}(z, H_q) \le \eta \right\}.
\end{equation}
Fix $\eta >0$ small enough such that $\mathcal N_p (\eta) \cap \mathcal N_q (\eta)$ stays at a positive
distance from $\overline D$.   Let
\begin{equation*}
\mathcal N_{p_k} (\eta) := \left\{ z\in \C^n: \mbox{dist}(z, H_{p_k}) \le \eta \right\}, \quad
\mathcal N_{q_k} (\eta) := \left\{ z\in \C^n: \mbox{dist}(z, H_{q_k}) \le \eta \right\}.
\end{equation*}
Then $\mathcal N_{p_k} (\eta) \cap \overline D$ converges to
$\mathcal N_p (\eta) \cap \overline D$ in the Hausdorff distance (and the same holds for $q_k$
and $q$ respectively).
So for $k$ large enough,
\[
(\mathcal N_{p_k} (\eta) \cap \overline D) \cap (\mathcal N_{q_k} (\eta) \cap \overline D) =\emptyset.
\]
Suppose now that $k$ is large enough so that $\d_D (x_k), \d_D (y_k) < \frac12 \eta$.
The affine real tangent hyperplane to $\partial D$ at $p_k$, $T_{p_k}$, is orthogonal to the real line passing through $x_k$ and $p_k$
and contains the affine complex hyperplane $H_{p_k}$.

Let $\pi'_k$ be the orthogonal projection
to the complex line through $x_k$ and $p_k$, parallel to $H_{p_k}$. Take a complex coordinate
on this line so that $p_k$ is represented by $0$ and the inward half line from $p_k$ containing $x_k$
goes to the positive imaginary half axis.
Then  $\pi'_k(\mathcal N_{p_k} (\eta))= \overline D(0,\eta)$,  $\pi'_k(x_k)=i \d_D (x_k)$,
and $\pi'_k(D) \subset \{\Im z >0\}=: \mathbb H$.
Since holomorphic maps contract the Kobayashi distance,
\begin{equation}
\label{halfhole}
k_D\left(x_k, D \setminus \mathcal N_{p_k} (\eta) \right)
\ge k_{\mathbb H}\left(i \d_D (x_k), \mathbb H \setminus \overline D(0,\eta)\right)
= \frac12 \log \frac1{\d_D (x_k)} - \frac12 \log \frac1{\eta},
\end{equation}
the last equality being obtained by an explicit computation.  An analogous inequality can be proved
for $k_D\left(y_k, D \setminus \mathcal N_{q_k} (\eta) \right)$ by using an orthogonal projection $\pi''_k$
parallel to $H_{q_k}$.

Finally, any curve from $x_k$ to $y_k$ must contain a point
$z_k \in F_k:=D \setminus \left(\mathcal N_{p_k} (\eta) \cup \mathcal N_{q_k} (\eta) \right)$. By the above estimates,
\begin{equation*}
k_D(x_k,y_k) \ge \inf_{z \in F_k} k_D(x_k,z)+\inf_{z \in F_k} k_D(z,y_k)
\ge
\frac12 \log \frac1{\d_D (x_k)} +\frac12 \log \frac1{\d_D (y_k)}-  \log \frac1{\eta},
\end{equation*}
and we are done.

The last statement follows at once from  the remark after Definition \ref{growth}.
\end{proof}

\begin{cor}
\label{trivface}
Let $D$ be a bounded convex domain, and let $p \in \partial D$ be such that
 $F_p=\{p\}$. Then for any $q \in \partial D \setminus \{p\}$,  \eqref{logest} holds.
% $\limsup_{(x, y)\to (p,q)} (x|y)_o < \infty$.

 In particular, if $F_p=\{p\}$ for any $p\in\partial D$ and $D$ has $\frac12$-log-growth,
 then $D$ has the visibility property.
\end{cor}

The above conclusion for  bounded convex domains with $\mathcal C^{1,\alpha}$-boundary follows as well from \cite[Theorem 4.1(2)]{Z3}.

One may ask whether a full converse of Proposition \ref{necvis} holds.  It does not, and
it seems that having complex faces on the boundary ``on the way'' from $p$ to $q$ is an obstruction.

\begin{prop}
\label{stdinvis}
Let $D:=\D^2$, $p:=(-1,0)$, $q:=(1,0)$. Then
the pair $\{p,q\}$ satisfies \eqref{stdest}, but not the visibility property.
\end{prop}
\begin{proof}
%  Let $p'\to p$, $q'\to q$.
  If $x, y$  are close enough to $p$
and $q$ respectively, $\d_D(x)=1-|x_1|$, $\d_D(y)=1-|y_1|$, and using the projection on $\D\times\{0\}$,
$$
k_D(x,y)\ge k_{\D} (x_1,y_1)= -\frac12 \log(1-|x_1|^2) -\frac12 \log(1-|y_1|^2) +O(1).
$$
 So \eqref{logest} holds,
which is equivalent to \eqref{stdest} in this case.

Consider the points $p_\eps:=(-1+\eps,0)$, $q:=(1-\eps,0)$.
Then a geodesic  from $p_\eps$ to $q_\eps$ is given by the line segment $[-1+\eps,1-\eps] \times \{0\}$.
But another one can be obtained in the following way: let $c>1$ be chosen such that
$k_{\D}(0, 1- \eps^{1/2}) < k_{\D}(1-\eps, 1- c\eps^{1/2})$ (this is possible since
$k_{\D}(0, 1- \eps^{1/2}) = -\frac14 \log \eps +O(1) = k_{\D}(1-\eps, 1- c\eps^{1/2}) +O(1)$).
Let $\tilde p_\eps:= (-1+ c\eps^{1/2}, 1- \eps^{1/2})$, $\tilde q_\eps:= (1- c\eps^{1/2}, 1- \eps^{1/2})$.
Let $\gamma_\eps$ be the curve made up of a geodesic from $p_\eps$ to $\tilde p_\eps$,
followed by the geodesic from $\tilde p_\eps$ to $\tilde q_\eps$ (a horizontal line segment),
followed by a geodesic from $\tilde q_\eps$ to $q_\eps$.  Since $k_{D}(z,w)= \max (k_{\D}(z_1,w_1), k_{\D}(z_2,w_2))$,
we see that
\begin{multline*}
\ell  (\gamma_\eps) = k_{\D}(-1+\eps,-1+ c\eps^{1/2})+ k_{\D}(-1+ c\eps^{1/2}, 1- c\eps^{1/2})+
k_{\D}( 1- c\eps^{1/2},1-\eps)
\\
 = k_{\D}( -1+\eps,1-\eps),
\end{multline*}
so $\gamma_\eps$ is a geodesic segment too. But $\gamma_\eps \subset \D^2 \setminus \overline D(0,1- c\eps^{1/2})^2$
avoids any compact set in the bidisc
as $\eps\to 0$, so the pair $\{p,q\}$ does not have the visibility property.
\end{proof}

It remains an open question whether one can find such an example with $\partial D$  smooth.

One might ask whether the hypothesis $F_p\cap F_q=\emptyset$ of  Proposition~\ref{suffstd} can be weakened. A natural weaker condition is to assume that  $p\not\in F_q$ (which implies that also $q\not\in F_p$). Under such a condition we can prove:

\begin{prop}
\label{weaker}
Let $D$ be a convex domain (not necessarily bounded),
and let $p,q \in \partial D$ be such that $q\not\in F_p$. Then
\[
\limsup_{x\to p,y\to q}\left( \frac12 \max\{\log\frac{1}{\d_D(x)}, \log\frac{1}{\d_D(y)}\} -k_D(x,y) \right)<\infty.
\]
\end{prop}

\begin{proof}
By definition of multiface, if $q\not\in F_q$ then also $p\not\in F_q$.

For a  point
$z \in D$, and a vector $v \in \C^n$, we let
\begin{equation}
\label{dirdist}
\d_D (z;v):= \sup\left\{ r>0 : z+\lambda v \in D \quad \forall \lambda\in \C, |\lambda|<r \right\}.
\end{equation}
It follows by the estimate \cite[p.~633]{NT} that
\[
k_D(x,y)\ge\frac12 \log\left(1+\frac{1}{\min\{\log\d_D(x;\frac{x-y}{\|x-y\|}),\log\d_D(y;\frac{x-y}{\|x-y\|})\}}\right).
\]
We claim that
there is a constant $c$ so that
 $\d_D(x;\frac{x-y}{\|x-y\|})\le c \d_D(x)$ and $\d_D(y;x-y)\le c \d_D(y)$, for $x$ close
 enough to $p$ and $y$ close to $q$, from which the statement follows at once.

In order to prove the claim,  suppose by contradiction that there exist $x_k\to p$,  $y_k\to q$, such that
  \begin{equation}
  \label{incid}
  \d_D(x_k) =o\left(\d_D\left(x_k;\frac{x_k-y_k}{\|x_k-y_k\|}\right)\right).
  \end{equation}
  Then $v_k:=\frac{x_k-y_k}{\|x_k-y_k\|}\to v:= \frac{p-q}{\|p-q\|}$.
  Let $p_k\in \partial D$ be such that $\|x_k-p_k\|= \d_D(x_k)$,
  and let $H_k$ be the supporting complex hyperplane at $p_k$---note that it is unique because
  $B(x_k,\d_D(x_k))\subset D$ and $p_k\in \partial B(x_k,\d_D(x_k))$.
  Passing to a subsequence if needed, we may assume that $H_k$ converges
  to a supporting complex hyperplane $H_\infty$ at $p$. Then,
  decomposing the vectors $v_k$ in components parallel and orthogonal to $H_k$,
   \eqref{incid} implies
  that $v\in H_\infty$. This means that
  the direction $p-q$ is parallel to $F_p$, and by convexity, that $q\in F_p$,
contradicting the hypothesis.

A similar argument shows that $\d_D(y;x-y)\le c \d_D(y)$, for $y$ close to $q$, and we are done.
\end{proof}

The conclusion of the previous proposition cannot be improved: consider the bidisc $D:=\D^2$, and $p:=(1,0)$,
$q:=(0,1)$, $x:=(1-\eps,0)$, $y:=(0, 1-\eps)$. Then $q\notin F_p$, $p\notin F_q$,
but $F_p\cap F_q=\{(1,1)\}$ and
\[
k_{\D^2} (x,y)= \frac12 \log\frac{2-\eps }{\eps} \le \frac12 \log\frac{1 }{\d_{\D^2}(x)} +\log 2
=\frac12 \log\frac{1 }{\d_{\D^2}(y)} +\log 2.
\]

\section{Points of infinite type}\label{infinite-type}

By  \cite{Z1}  bounded smooth convex finite D'Angelo type domains are Gromov hyperbolic and so they have the visibility property by \cite{BGZ}. In this section we first prove Theorem~\ref{onept}, proving that, although not Gromov hyperbolic, smooth bounded convex domains whose boundary points are of finite type, except at most finitely many, have the visibility property.

To this aim, we need to set up some notation and preliminary notions.

In this section,  $D\subset \C^n$ is a bounded convex domain with smooth boundary.

Since any bounded domain with $\mathcal C^2$ boundary
satisfies the inner ball condition,
there exists $\d_c>0$ such that for any $z\in D$ with $\d_D (z) < \d_c$,
there exists a unique point $\zeta \in \partial D$ such that $\|z-\zeta\|=\d_D(z)$.
We denote $\zeta = \pi(z)$. The fiber $\pi^{-1}(\pi(z))$ is a subset of the real normal line
to $\partial D$ at $\pi(z)$. The map $z\mapsto (\pi(z), \d_D(z))$ is a diffeomorphism
from $\{z\in D: \d_D(z)<\d_c\}$ to $\partial D \times (0, \d_c)$.

Using the inner ball condition and an affine mapping,
and taking $\d_c$ smaller as needed, it is easy to see that
there exists a constant $A_1$ such that for all $p, q\in D$ such that
 $\pi( p )=\pi(q)$ and $\d_c \ge \d_D(p )\ge \d_D(q)$, then
\begin{equation}
\label{alongnormal}
k_D(p,q) \le \frac12 \log \frac{\d_D(p )}{\d_D(q)} + A_1.
\end{equation}

Note also that $K_c := \left\{ z \in D: \d_D (z) \ge \d_c\right\}$
is a compact set, and so bounded in the Kobayashi distance of $D$.

We recall the definition of a minimal basis at a point $z$,
and of the directional distances to the boundary $\tau_j(z)$, as given in \cite{NT}.

Let $z\in D$ and let $z^1:= \pi(z)$ and $\tau_1(z):=\lVert z^1-z\rVert= \d_D(z).$
Let
$H_1=z+(\C(z^{1}-z))^{\bot}$, where, for a complex  space $V\subset \C^n$ we let $V^\bot$ denotes the complex space orthogonal (with respect to the standard Hermitian product) to $V$. Let  $D_1=D\cap H_1.$ Note that $D_1$ is a convex domain in the affine complex space $H_1$ of dimension $n-1$ and $z\in D_1$.

Let
$z^2\in\partial D_1$ so that $\tau_2(z):=\lVert z^{2}-z\rVert=
\d_{D_1}(z).$ Let $H_{2}=z+(\textup{span}_\C(z^{1}-z,z^{2}-z))^{\bot},$
$D_{2}=D\cap H_{2}$ and iterate the construction for $n$ steps. Thus we get an orthonormal basis of
the vectors $ e_j=\frac{z^j-z}{\lVert z^{j}-z\rVert},$ $1\le j\le
n,$ which is called \emph{minimal} for $D$ at $z$, and positive numbers
$\tau_{1}(z)\le \tau_{2}(z)\le\dots\le \tau_{n}(z)$ (the basis and
the numbers are not uniquely determined).

From the last formula in the proof of  \cite[Lemma 1.3]{Gau} (p. 379, line 4) we have

\begin{lemma}
\label{tauest}
Let
$S\subset \partial D$ be  compact and let $U_1\subset \partial D$ be a neighborhood of $S$. Suppose there exists $M>0$ such that the type of $\zeta\in U_1$ is bounded from above by $M$.
Then there exist a neighborhood $U_2\subseteq U_1$ of $S$,  $\tilde \d >0$
and $C_1 >0$ such that for every $z\in D$ with $\d_D(z) \le \tilde \d$ and
$\pi(z)\in U_2$, we have  $\tau_j(z) \le C_1 \d_D(z)^{1/M}$ for all $1\le j \le n$.
\end{lemma}

The following lemma is an immediate consequence of \cite[Theorem 1, (ii)]{NT}.

\begin{lemma}
\label{balls}
If $q,z \in D$ are such that $k_D(q,z)<r$, then
$$
\max_{1\leq
j\leq n}\frac{|z_j-q_j|}{\tau_j(q)}< e^{2r}-1,
$$
where $(z_1,\ldots, z_n)$ and $(q_1,\ldots, q_n)$ are the coordinates of $z$ and $q$ in the affine coordinates system $\{O; e_1,\ldots, e_n\}$, where $O$ is the origin in $\C^n$ and $\{e_1,\ldots, e_n\}$ is a minimal orthonormal basis for $D$ at $q$.
\end{lemma}

We now prove that if two points with comparable distance from the
boundary are joined by a geodesic that stays as close to the boundary as they are,
then they must be close to each other in the Euclidean distance.
This is achieved by comparing the Kobayashi length of a curve staying close enough to
the boundary with an upper bound for the Kobayashi distance between its extremities.

\begin{lemma}
\label{sameheight}
Let
$S\subset \partial D$ be  compact and let $U_1\subset \partial D$ be a neighborhood of $S$. Suppose there exists $M>0$ such that the type of $\zeta\in U_1$ is bounded from above by $M$. Let $\d \in (0,\min\{\d_c, e^{-1}\})$.
Suppose $p,q \in D$ and $\d_D(p ) , \d_D(q )
\in [\d/2 , \d ]$. Let $\gamma$ be a geodesic joining $p$ and $q$ such that
$\d_D(\gamma(t)) \le \d$ and
$\pi(\gamma (t)) \in S$ for any $t$.
Then there exist constant $A_2, C_2>0$, depending only on $S$ and $D$, such that
$$
\left\|p-q \right\| \le C_2 \d^{1/M} \left( \log \frac1\d + A_2 \right).
$$
\end{lemma}

\begin{proof}
In what follows, $C$ stands for a constant whose value may change from line to line.

Consider two points $z,w$ such that $\pi(z), \pi(w) \in S$ and
$k_D (z,w) \le 1$. Then Lemma \ref{balls} implies that
$|z_j-w_j| \le (e^2-1)\tau_j(z)$ for $1\le j \le n$ (where $(z_1,\ldots, z_n)$ and $(w_1,\ldots, w_n)$ are the coordinates of $z$ and $w$ in a minimal orthonormal basis for $D$ at $z$).

Then, by Lemma \ref{tauest},
\begin{equation}
\label{kobaest}
\left\| z - w \right\| \le \sqrt{n} \max_{1\le j \le n} |z_j-w_j|
\le C \max_{1\le j \le n} \tau_j(z)
\le C \d_D (z)^{1/M} \le C \d^{1/M}.
\end{equation}

In particular, if $k_D(p,q)\leq 1$, the result holds whenever $C_2\geq C$ and $A_2\geq 0$ (since $\log \frac1\d>1$ 
because $\d<e^{-1}$).

We can thus assume that $k_D(p,q)> 1$. Write $k_D(p,q)=m+s$, with $m\in \mathbb N$, $m\geq 1$ and $s\in [0,1)$, and assume $\gamma(0)=p$ and $\gamma(m+s)=q$.   In particular, the hyperbolic length $\ell(\gamma)$ of $\gamma$ is greater than or equal to $m$. Since $k_D(\gamma(k+1), \gamma(k))=1$ for $k=0,\ldots, m-1$ and $k_D(\gamma(m), \gamma(m+s))\leq 1$, by \eqref{kobaest}, we have for $k=0,\ldots, m-1$,
\[
\left\| \gamma(k+1) - \gamma(k) \right\| \le C \d^{1/M},
\]
and $\left\| \gamma(m+s) - \gamma(m) \right\| \le C \d^{1/M}$.
By the triangle inequality,
\begin{equation}\label{stim-12}
\left\|p-q \right\| \le (m+1) C \d^{1/M}\leq (\ell(\gamma)+1)C \d^{1/M}.
\end{equation}

Now, let $p_0, q_0\in D$ be such that
$\pi ( p_0 )=  \pi ( p )$ and $\pi ( q_0 )=  \pi ( q )$ and $\d_D (p_0)=\d_D (q_0)=\d_c$.  Since $\gamma$ is a geodesic, and using \eqref{alongnormal}, we have
$$
\ell (\gamma) \le k_D (p,p_0) + k_D (p_0,q_0) +  k_D (q_0,q)
\le \log \frac{2\d_c}{\d} + 2A_1 + \mbox{diam} K_c,
$$
where $\mbox{diam} K_c$ denotes the diameter in the Kobayashi distance of $K_c$, which is finite.

Putting together the previous inequality and \eqref{stim-12} we have the result.
\end{proof}

Now we are in a position to prove Theorem~\ref{onept}.

\begin{proof}[Proof of Theorem \ref{onept}]
Let $p, q\in \partial D$. 
We aim to show that there exists $\d_{p,q}>0$ such that for any $\{p_k\}, \{q_k\}\subset D$ converging to $p$ and $q$ respectively, for $k$ large
enough, any geodesic from $p_k$ to $q_k$ intersects $\{ z\in D: \d_D (z)
\ge \d_{p,q}\}$.

We argue by contradiction and suppose there is no such $\d_{p,q}$.
Let $\gamma_k$ be a geodesic such that $\gamma_k(0)=p_k$ and $\gamma_k(R_k)=q_k$ for some $R_k>0$.
Then, for any $\d>0$ we can find $k_\d$  so that $\d_D(\gamma_k(t)) \le \d$ for any $t$ and $k\geq k_\d$.
\smallskip

{\it Case 1.}  Either $q$ or $p$ (or both)  is of finite type.
\smallskip

Without loss of generality, we may assume that $q$ is a point of finite type.  
For $r$ small enough,  any point in $S:=B(q,r)\cap \partial \Omega$ is of finite 
type bounded by $M$, and $p\notin \bar S$. Let $U:= \pi^{-1}(S) \subset \Omega\setminus K_c$. 
For $k$ large enough, $q_k\in U$.  Given a geodesic $\gamma_k$, 
by the choice of $S$, for $k$
large enough, $\gamma_k([0,1])\not \subset \bar U$. Let 
$$t_p:=\sup\{t \in [0,1): \gamma_k(t)\notin U\}$$ 
and $p'_k:=
\gamma_k(t_p)$. Then $\gamma_k|_{[0,t_p]}$ is a geodesic, and we have reduced ourselves 
to the case where $\gamma_k([0,1]) \subset \bar U$, $\|p_k-q_k\|\ge r/2$.

We make a further reduction by considering $t_m$ such that 
$$
\delta_\Omega(\gamma_k(t_m))
=\max_{t\in [0,1]} \delta_\Omega(\gamma_k(t)). 
$$
Then either $\|p_k-\gamma_k(t_m)\|\ge r/4$
or $\|q_k-\gamma_k(t_m)\|\ge r/4$, so by restricting the domain of
$\gamma_k$ again we may assume that we have a geodesic, denoted again
by $\gamma_k$, from a point $p_k$ to a point $q_k$, with $\|p_k-q_k\|\ge r/4 $,
such that for all $t\in[0,1]$, 
$$
\delta \ge \delta_0:= \delta_\Omega(p_k) =  \delta_\Omega(\gamma_k(0)) \ge \delta_\Omega(\gamma_k(t)).
$$
Let $t_0=0$ and for any integer $j \le \log_2 \left( \delta_\Omega(p_k)/\delta_\Omega(q_k) \right)$,
$$
t_j:= \sup\left\{ t\in [0,1): \delta_\Omega(\gamma_k(t)) \ge 2^{-j} \delta_0 \right\},
$$
so that $t_0< \dots <t_j<t_{j+1}<\dots <t_s <1=: t_{s+1}$, for some integer $s\ge 0$,
and, since $\gamma_k$ is continuous, $\delta_\Omega(\gamma_k(t_j)) =  2^{-j} \delta_0 $.

Then Lemma \ref{sameheight} implies that for $j\le s$,
$$
\left\| \gamma_k(t_j) - \gamma_k(t_{j+1}) \right\|
\le 
C_2 \delta_0^{1/M} 2^{-j/M} \left( \log \frac{2^{j}}{\delta_0} + A_2 \right),
$$
so that 
$$
\frac{r}4 \le \|p_k-q_k\|
\le C'_2 \delta_0^{1/M}  
\left( \log \frac1{\delta_0} + A_2 \right) \sum_{j=0}^s 2^{-j/M} (j + 1).
$$
The last sum is bounded by a constant independent of $s$, and the factor involving 
$\delta_0$ can be made arbitrarily small by reducing $\delta$, so we arrive at the
desired contradiction since $r$ depends only on $p$ and $q$. 

\smallskip
{\it Case 2.} $p$ and $q$ are of infinite type.
\smallskip

Let $E\subset \partial D$ be the set of points of infinite type different from $p$ and $q$. By hypothesis, $E$ is finite. For $r>0$ let $E_r:= \{ z\in \C^n: \mbox{dist}(z,E) < r\}$.

Let $r>0$ be such that
$\overline{B(p,r)} \cap \overline{B(q,r)} =\emptyset$  and  $(\overline{B(p,r)} \cup \overline{B(q,r)}) \cap \overline{E_r}=\emptyset$.
Then $U_1:=\partial D \setminus \left( \overline{B(p,r/2)} \cup \overline{B(q,r/2)} \cup \overline{E_{r/2}}\right)$ is an open set in $\partial D$ formed by points of type at most $M$ for some $M>0$ (since the type of a point is a locally bounded function). Let $S:=\partial D\setminus  \left( B(p,r) \cup B(q,r) \cup E_{r}\right)$. Then $S\subset U_1$ is a compact set, and its
points are at distance at least $r/2$ from $E \cup \{p,q\}$.

Now we can argue as before. Given two sequences $\{p_k\}$ and $\{q_k\}$ in $D$ converging to $p$ and $q$ respectively, if $\gamma_k$ is a geodesic joining $p_k$ and $q_k$ such that for every $\d>0$ there exists $k_\d$ so that $\d_D(\gamma_k(t))\leq \d$ for all $k\geq k_\d$, we can easily find an interval $[t_k, s_k]$ such that $\pi(\gamma_k(t))\in S$ for all $t\in [t_k, s_k]$ and $\|\gamma(t_k)-\gamma(s_k)\|\geq c$ for some constant $c>0$ (and $k$ large enough), getting a contradiction by
using the estimates above.
\end{proof}

In case the set of points of infinite type is not finite, the (bounded convex, smooth domain) $D$ might or might not have the visibility property. Roughly speaking, the visibility property  depends in a critical way on the rate of approach of the boundary to its complex tangent space. We give  an example below.

Let $\psi, \chi_1, \chi_2 \in \mathcal C^\infty (\R,\R_+)$ be even functions,
 strictly increasing on $\R_+$.
Let $\Omega_\psi$ be a convex domain such that
\begin{equation}
\label{domcplxtgt}
\begin{split}
&\Omega_\psi \cap \{ \|z\| <3 \}  =
\\
&\left\{ z \in \C^2: \|z\| <3, \Re z_2 >
 \psi (\Re z_1) + \chi_1 ( (|\Im z_1| - 2)_+) + \chi_2(\Im z_2)
\right\}.
\end{split}
\end{equation}
We also assume that $\Omega_\psi$ is smoothly bounded and contained in the ball $B(0,5)$.
Notice that $\partial \Omega_\psi \cap \C\times \{0\} = [-2i,2i]\times \{0\}$.
We choose $p:= (i, 0)$, $q := (-i, 0)$. By choosing $\chi_1, \chi_2$ convex,
we can ensure that $\partial \Omega_\psi$ is strictly pseudoconvex away from the line segment $\{0\} \times [-2i;2i]$.

If  for some $\epsilon>0$
$$\int^\epsilon_0 \psi^{-1}(u) \frac{du}u < \infty,$$
 the domain $\Omega_\psi$
satisfies the Goldilocks condition from Definition \ref{gold}, and therefore has the visibility property,
by \cite[Theorem 1.4]{BZ}.   This is achieved, for instance, when
$\psi(x) \ge \exp\left( -\frac1x (\log \frac1x)^{-\alpha} \right)$, for some $\alpha>1$,
because then
\[
 \psi^{-1}(u) \le \frac1{\left(\log \frac1u\right) \left(\log \log \frac1u \right)^\alpha  }.
 \]
 In contrast, taking $o$ some inner point in $\Omega_\psi$, we have
\begin{prop}
\label{bdrygeod}
If $\psi(x)= o\left( \exp\left( - \frac\pi{2x}  \right) \right)$ near $ x=0$, then
$\limsup_{(p', q')\to (p,q)} (p'|q')_o = \infty$, and
therefore $\Omega_\psi$ does not have the visibility property.
\end{prop}

\begin{proof}
Let $p_\eps := (i, \eps)$, $q_\eps := (-i, \eps)$, which tend to $p$ and $q$ respectively as $\eps\to 0$. We claim that
\begin{equation}\label{Eq-bound-Kob-spec}
k_{\Omega_\psi} (p_\eps,q_\eps) \le \log 2 + \frac\pi{2\psi^{-1}(\eps)}.
\end{equation}

In order to prove \eqref{Eq-bound-Kob-spec}, we will bound $k_{\Omega_\psi} (p_\eps,q_\eps)$ by constructing an almost explicit analytic disc containing $p_\eps$ and $q_\eps$.
Let $\omega_\eps:= \{ \zeta \in \C: (\zeta,\eps) \in {\Omega_\psi}\}$. For $\eps$ small
enough, by convexity, $(\C \times \{\eps\}) \cap {\Omega_\psi} \subset \{ \|z\| <3 \}$, so for those values
of $\eps$,
$(\zeta,\eps) \in {\Omega_\psi}$
if and only if
$$
\eps >\psi \left(|\Re \zeta|\right) + \chi_1 \left( (|\Im \zeta| - 2)_+ \right) ,
$$
and, in particular,
$$
\omega_\eps \supset \left\{ |\Im \zeta| < 2, \psi (\Re \zeta) < \eps \right\}=:R_\eps.
$$
Let $\eps':= \psi^{-1}(\eps)$ and
$$
R'_\eps:= \frac{i}{\eps'} R_\eps = \left\{ |\Re \zeta| <\frac{2}{\eps'}, |\Im \zeta| < 1 \right\}.
$$
Then
$$
k_{\Omega_\psi} (p_\eps,q_\eps)\le k_{\omega_\eps }(i,-i) \le k_{R_\eps }(i,-i)
=k_{R'_\eps } \left(\frac{1}{\eps'},-\frac{1}{\eps'}\right).
$$
The conformal map $\phi(z):= \frac{e^{\frac\pi2z}-1}{e^{\frac\pi2z}+1}$
maps the strip $\{|\Im \zeta| < 1 \}$ to the unit disc. The line segments $ \left\{ \Re \zeta =\pm \frac{2}{\eps'} \right\}$
are mapped to arcs of circles perpendicular to the unit circle which intersect the diameter $(-1,+1)$
at the points $\frac{e^{\frac\pi{\eps'}}-1}{e^{\frac\pi{\eps'}}+1}$ and
$\frac{e^{-\frac\pi{\eps'}}-1}{e^{-\frac\pi{\eps'}}+1}$, so that
$$
\phi\left( R'_\eps \right) \supset D\left( 0, 1- 2 e^{-\frac\pi{\eps'}}\right),
$$
while $\phi\left(\pm\frac{1}{\eps'}\right)= \pm \frac{1-e^{-\frac\pi{2\eps'}}}{1+e^{-\frac\pi{2\eps'}}}$.

By renormalizing the smaller disc, for $\eps'$ small enough,
$$
k_{R'_\eps } \left(\frac{1}{\eps'},-\frac{1}{\eps'}\right)
\le k_{\D} \left(1-e^{-\frac\pi{2\eps'}}, -1+e^{-\frac\pi{2\eps'}}\right)
= 2 k_{\D} \left(1-e^{-\frac\pi{2\eps'}},0\right)
\le \log 2 + \frac\pi{2\eps'},
$$
and \eqref{Eq-bound-Kob-spec} is proved.

To conclude the proof, notice that
\[
\frac12 \log \frac1{\d_{\Omega_\psi}(p_\eps)} + \frac12 \log \frac1{\d_{\Omega_\psi}(q_\eps)} = \log \frac1\eps,
\]
and our hypothesis says that $\eps = o\left( \exp\left( - \frac\pi{2 \psi^{-1}(\eps) }  \right) \right)$,
i.e. $\log \eps + \frac\pi{2 \psi^{-1}(\eps) } \to -\infty$ as $\eps\to0$. By \eqref{Eq-bound-Kob-spec}, this implies
\[
-\frac12 \log \frac1{\d_{\Omega_\psi}(p_\eps)} - \frac12 \log \frac1{\d_{\Omega_\psi}(q_\eps)}
+ k_\Omega (p_\eps,q_\eps)  \le \log \eps + \frac\pi{2 \psi^{-1}(\eps) } \to -\infty,
\]
and we are done.
\end{proof}

Note that the previous condition is analogous to the ``log-type" condition in \cite{LW}, but slightly less demanding.

\section{Localization}
\label{localization}

The  results from Section \ref{visgro} are  obtained with the
help of a global hypothesis, namely convexity of the whole domain.  It seems natural to look for
analogues with suitably localized hypotheses.

To this aim we introduce and study two objects: the {\sl $k$-points} and the {\sl locally $\C$-strictly convex point}.

\subsection{$k$-points}

\begin{defn}
\label{kpt}
Let $D$ be a domain in $\C^n.$
We say that $p\in\partial D$ is  a \emph{ $k$-point} if for every neighborhood $W$ of $p$,
$$
\liminf_{z\to p}[k_D(z,W^c)+\frac12 \log \d_D(z)]>-\infty,
$$
where $\ds k_D(z,W^c):=\inf_{w\in D\setminus W}k_D(z,w).$
\end{defn}

If all boundary points are $k$-points, in many cases, the domain enjoys the visibility
property.  We recap the consequences of previous results.

\begin{prop}
\label{kpteq}
Let $D$ be a bounded convex domain  in $\C^n$.
\begin{enumerate}[(i)]
\item
If for any $q \in \partial D \setminus \{p\}$,  $\limsup_{(x, y)\to (p,q)} (x|y)_o < \infty$,
then $p$ is a $k$-point.
\item
If, furthermore, $\partial D$ is Dini-smooth, then the following are equivalent:
\begin{enumerate}[(a)]
\item
any $p\in \partial D$ is a $k$-point;
\item
for any $p\neq q \in \partial D$,
$\limsup_{(x, y)\to (p,q)} (x|y)_o < \infty$;
\item
$D$ enjoys the visibility property.
\end{enumerate}
\end{enumerate}
\end{prop}

\begin{proof}
(i) Assume $p$ is not a $k$-point, then we can choose sequences $\{x_k\}\to p$ and $\{y_k\} \subset W^c$
such that $k_D(x_k,y_k) + \frac12 \log \delta_D(x_k) \to -\infty$.  Passing to a subsequence, we may
assume $y_k\to q\in \overline D\setminus W$. If $q\in D$, then the assumption contradicts \eqref{dumbest}.
Suppose $q\in \partial D\setminus W$. Since $D$ is convex, the assumption in (i)
implies \eqref{logest}, so for $k$ large enough,
$$
k_D(x_k,y_k) + \frac12 \log \d_D(x_k) \ge \frac12 \log \frac1{\d_D(y_k)} -C,
$$
which shows that $p$ is a $k$-point.

(ii) (b) implies (a) follows at once from  (i). Conversely,
if both $p, q$ are $k$-points, and $p\neq q$, then selecting disjoint neighborhoods $W_p$ and $W_q$,
the Kobayashi length of any curve connecting $p'\in W_p$  to $q'\in W_q$ is greater
than $k_D(p',W_p^c) + k_D(q',W_q^c)$ and so
we see that the pair $\{p,q\}$ satisfies  \eqref{logest},
therefore by Dini-smoothness, \eqref{stdest},  hence, (a) implies (b).

Finally, properties (b) and (c) are equivalent by Proposition \ref{suffvis}.
\end{proof}

The proof of (ii) also implies:

\begin{prop}
\label{prop:Dini-k-vis}
Let $D$ be a bounded complete hyperbolic domain  in $\C^n$. If $\partial D$ is Dini-smooth and any $p\in \partial D$ is a $k$-point, then $D$ enjoys the visibility property.
\end{prop}

Proposition \ref{bdrygeod} provides examples of points in the boundary of a convex domain,
 with a complex face reduced to a real line segment,
which are not $k$-points.

Note that a $k$-point $p\in \partial D$  satisfies the following necessary condition: for every neighborhood $W$ of $p$,
\begin{equation}\label{Eq:weak-k-point}
\limsup_{z\to p} k_D(z,W^c)=\infty.
\end{equation}

With this remark at hand, we can prove

\begin{lemma}
\label{discbdry}
Let $D\subset \C^n$ be a bounded domain. If  $\partial D$ is $\mathcal C^1$-smooth, or if $D$ is convex,  and $p$ lies in the interior of an
affine analytic disc contained in $\partial D$, then \eqref{Eq:weak-k-point} is not satisfied and, in particular, $p$ is not a  $k$-point.
\end{lemma}

This is proved in the case where  $\partial D$ is  $\mathcal C^1$ by \cite[Proposition 4.6]{Z2}.
The converse is true for any $\C$-convex domain, by \cite[Proposition 3.5]{Z2},
which proves a stronger fact: if \eqref{Eq:weak-k-point} does not hold then $p$ sits in the
interior of a convex face.

\begin{proof}
It is enough to ``push'' the analytic disc on $\partial D$ inside the domain $D$ by an arbitrarily small amount.
Under each of the hypotheses, we will prove a slightly stronger property of the boundary:

(PB) for any $p\in \partial D$, there exist a unit vector $v\in\C^n$, a neighborhood $U$
of $p$, and a positive number $\eps$ such that :  for any $z\in U\cap \overline D$,
any $t\in (0,\eps)$, then $z+tv\in D$.

If (PB) holds, let $E$ be a complex affine disc centered at $p$, of radius $r_0$ small enough
so that $E\subset \partial D\cap U$. For $t<\eps$, we have a family of complex affine discs parallel to $E$,
$E+tv$. Take $W:= B(p, r_0/2)$, the ball of radius $r_0/2$ around $p$.
Using the analytic discs $E+tv$ we see that $k_D (p+tv, \partial W ) $ remains bounded
as $\eps\to 0$.

 Any domain with $\mathcal C^1$ boundary satisfies (PB) (this is almost
stated in  \cite[Remark 3.2.3. (b)]{JP}, and elementary).

In the case where $D$ is convex with no additional boundary smoothness, if $p\in \partial D$,
consider any $p_0 \in D$. There is an $r>0$ such that the closed ball $\bar B(p_0,r)\subset D$. Consider
$F:=\left( \bar B(p_0,r/2) + \R_+ (p-p_0)  \right) \cap \partial D$. This is a closed neighborhood of $p$
in $\partial D$. We can take for $U$ a small enough open neighborhood of $F$, $v:= \frac{p_0-p}{\|p-p_0\|}$,
 and $\eps< \mbox{dist}(F, \bar B(p_0,r/2))$.
\end{proof}

Taking into account that by \cite[Proof of Proposition 2.4]{Sib} (see  \cite[Proposition 10]{NPZ}
for more details), a point $p\in \partial D$  lies in no non-trivial affine analytic disc
 contained in $\partial D$ if and only if
$p$ is a holomorphic peak point for $A(D)$, we have:

\begin{cor}
\label{peak}
If $D$ is convex and bounded, $p\in \partial D$, then
 $p$ satisfies \eqref{Eq:weak-k-point}  if and only if $p$ is a peak point for $A(D)$.
 \end{cor}

On the other hand, from Corollary \ref{trivface} and the proof of Proposition \ref{kpteq}(i)  we have

\begin{prop}
If $D$ is a bounded convex domain, $p\in \partial D$ and $F_p=\{p\}$, then $p$ is a $k$-point.
\end{prop}

Our next aim is to localize the previous results. We start with:

\begin{thm}
\label{kp}
Let $D$ be a bounded domain in $\C^n$ and $p\in\partial D$.
Let $U$ be a neighborhood of $p$ such that $p$ is a $k$-point for $D\cap U$
and assume that $D\cap U$ has $\a$-growth for some $\a<1$. Then $p$ is a $k$-point for $D$.
\end{thm}

Since bounded convex domains have $\a$-log-growth for some
$\a>0$ (see the remark after Definition \ref{growth}), we have the following corollary:

\begin{cor}
\label{ckp}
Let $D$ be a bounded domain in $\C^n.$ Let $U$ be a neighborhood of $p$ such that $D\cap U$ is convex and $p$ is a
$k$-point for $D\cap U$. Then $p$ is a $k$-point for $D$.
\end{cor}

\begin{proof}[Proof of  Theorem \ref{kp}]
In the sequel, when $U$  and $D$ are  open sets, we define $D_U:=D\cap U$.

Before starting the actual proof of Theorem~\ref{kp}, we need a lemma to compare  the  Kobayashi-Royden metrics
of $D$ and $D\cap U$ when $z$ is near $p$.

\begin{lemma}
\label{deltapower}
Given any $c<1$, there exist $c'>0$ and a neighborhood $W$ of $p$, $W\subset \subset U$, such
that for any $z\in W$,
\begin{equation}
\label{loci}
\k_D(z;X)\ge(1-c'\d_D(z)^c)\k_{D_U}(z;X),\ X\in\C^n.
\end{equation}
\end{lemma}

\begin{proof}
Let $V\subset \subset U.$ Since $D$ is bounded, $\min_{z\in V} l_D(z,U^c) =: c_0 >0$.
By the  localization formula (see \cite[Lemma 2]{Roy}):
\begin{equation}
\label{roy}
\k_D(z;X)\ge l_D(z,U^c)\cdot \k_{D_U}(z;X),
\end{equation}
where
$l_D(z,U^c)= \inf_{w\in D\setminus U} l_D(z,w)$, so for $z\in D_V,$ we have
\begin{equation}
\label{lock}
\k_D(z;X)\ge c_0 \k_{D_U}(z;X),\ \ X\in\C^n.
\end{equation}
Since $k_D$ is the integrated form of $\k_D,$
$$
k_D(z,V^c)\ge c_0 k_{D_U}(z,V^c),\ z\in D_V.
$$
We may choose $W\subset \subset V$
such that for $z\in W$, $\delta_{D_U}(z)=\delta_D(z)$. Since $p\in\partial D_U$ is a $k$-point,
for $ z\in D_W, $ we have
%and $c'>0$ such that
$$
l_D(z,U^c)\ge l_D(z,V^c)\ge\tanh k_D(z,V^c)\ge \tanh c_0 k_{D_U}(z,V^c)\ge 1-c'\delta_D(z)^{c_0},
$$
where $c'$ depends on the constant implicit in Definition \ref{kpt}.

Shrinking $W$ further, we may reduce $\sup_{z\in  W}\delta_D(z)$ so that \eqref{lock} holds
for $z\in W$ with any constant $c<1$ we choose in the place of $c_0$. Repeating the previous argument once
more, we get \eqref{loci}, and the Lemma is proved.
\end{proof}

Now, by \eqref{lock}, there exists $W$ a neighborhood of $p$, and a constant $\hat c>0$ such that
$$\k_D(z;X)\ge\hat c||X||,\ z\in D_W,\ X\in\C^n.$$

Let now $z\in D_W$ and $w_1\not\in D_W.$ %By \cite[Proposition 4.4]{BZ},
Let $\gamma:[0,L_1]\rightarrow D$ be a $\mathcal C^1$ curve from $z$ to $w_1$ such that
\[
\int_0^{L_1} \k_D(\gamma_1(s);\gamma_1'(s)) ds \le k_D (z,w_1)+\eps.
\]
Let $L_2:=\inf\{s \in [0,L_1]: \gamma_1(s)\notin W\}$ and $w:=\gamma_1(L_2)\in D\setminus W$. It will
be enough to bound from below $k_D(z,w)$.

We can then apply the following lemma, which is essentially proved in \cite[Proposition 4.4]{BZ} (``existence
of  $(1,\eps)$-almost geodesic'').

\begin{lemma}
\label{abs}
Let $D$ be a domain in $\C^n$ and
$\g:[0,1]\to D$ be a $\mathcal C^1$-curve with non-singular points.
Assume that there exists a constant $\hat c>0$ such that
$\k_D(\g(t);\g'(t))\ge \hat c||\g'(t))||$ for any $t\in[0,1].$ Let $\s:[0,L]\to D$, $s\mapsto \s(s)$,
be the parametrization of $\g$ by Kobayashi-Royden length, $ds = \k_D(\g(t);\g'(t)) dt$. Then $\s$ is an absolutely
continuous curve and $\k_D(\s(s);\s'(s))=1$ for almost every $s\in[0,L].$
\end{lemma}

By Lemma \ref{abs} applied to $\g:=\g_1|_{[0,L_2]}$,
for any $\var>0$ there exists %an $(1,\var)$-almost geodesic for $k_D$,
a curve
$\s:[0,L]\to D$, with $\s(0)=z$ and $\s(L)=w$, such that  $\k_D(\s(s);\s'(s))=1$ for almost every $s\in[0,L]$
and $L\le k_D(z,w) +\eps$.

%Set $L'=\inf\{t\in(0,L]:\s(t)\not\in D_W\}$ and $w'=\s(L').$

Choose $t_0$ to maximize $\d_D(\s(t)),$ $t\in[0,L].$
Since $D_U$ has $\a$-growth, it follows % as in the proof of \cite[Lemma 3.3]{LW}
that
\begin{multline}
\label{deltaest}
|t-t_0|-\var < k_D(\s(t),\s(t_0))
\le k_{D_U}(\s(t),z_0)+k_{D_U}(\s(t_0),z_0)\\
\le \b \d_D(\s(t))^{-\a} + \b \d_D(\s(t_0))^{-\a} \le 2 \b \d_D(\s(t))^{-\a}.
\end{multline}

Define $f(t)=2\beta$ if $|t-t_0|<1+\var$ and $f(t)=2\beta/(|t-t_0|-\var)$ otherwise.
Shrinking $W$ as needed, we may assume that
$\d_D(\s(t))^\a\le f(t)$.

Then using \eqref{loci}
\begin{multline*}
k_{D_U}(z,w)\le \int_0^{L'} \k_{D_U} (\s(t),\s'(t)) dt
\\
<\int_0^{L}(1+2c'\d_D(\s(t))^c)\k_D(\s(t);\s'(t))dt
\\
=\int_0^{L}(1+2c'\d_D(\s(t))^c)dt \le
L+2c'\int_0^\infty f(t)^{\nu}dt=k_D(z,W^c)+\var+ C_\eps,
\end{multline*}
where $\nu=c/\a$
and $C_\eps= 4c' (2\beta)^\nu \left(1+\eps+\frac1{\nu-1}\right)$. %depends on $c, c',  \a$ and $\beta$, but does not depend on $k_D(z,W^c)$.

Since $\var>0$ was arbitrary, it follows that
$$k_{D_U}(z,W^c)\le k_D(z,W^c)+C_0,\ z\in D_W$$
which completes the proof of Theorem \ref{kp}.
\end{proof}

\begin{rem}
In Theorem~\ref{kp} one can remove the hypothesis that $D$ is bounded by assuming that $p$ is an antipeak point of $D$. Indeed, under such hypothesis,  \eqref{lock} follows from \cite[Proposition 7]{FN}.
\end{rem}

\subsection{Locally $\C$-strictly convex points}

\begin{defn}
\label{cvxble}
We say that $p\in\partial D$ is a \emph{locally $\C$-strictly convex point}
if there exists a biholomorphism $\Psi$ from a bounded open neighborhood $U$ of $p$ to $\Psi(U)$
such that  $\Psi(D\cap U)$
 is convex and
$F_{\Psi(p)}= \{\Psi(p)\}$, where the multiface is taken with respect to $\Psi(D\cap U)$.
\end{defn}

Note that any $\mathcal C^2$-smooth strictly pseudoconvex point is locally $\C$-strictly convex.

\begin{thm}
\label{se} Let $D$ be a domain in $\C^n$. Any  locally $\C$-strictly convex point  $p\in \partial D$ is a $k$-point
for $D$.
As a consequence, any pair of distinct locally $\C$-strictly convex points satisfies
the log-estimate \eqref{logest}.
\end{thm}

Since strictly pseudoconvex points are locally $\C$-strictly convex, this provides a (modest) generalization
of \cite[Corollary 2.4]{FR}.

\begin{proof}[Proof of Theorem~\ref{se}]Here we show how to adapt the proof of Theorem \ref{kp} under this generalized hypothesis.

By Proposition \ref{kpteq}, $\Psi(p)$ is a $k$-point for $\Psi(D\cap U)$. As $\Psi$ is a biholomorphism
it only changes the distances to the boundary up to a fixed multiplicative constant near $p$, and the
Kobayashi distances are invariant, so $p$ is a $k$-point for $D\cap U$, and thus a peak point as well.

Furthermore, $\Psi(D\cap U)$ has the $\a$-growth property by convexity, hence, $D\cap U$ has the $\a$-growth property near $p$. This is
exactly what is needed to complete the last part of the proof after Lemma \ref{abs}.

The second statement follows from the proof of Proposition \ref{kpteq} (ii).
\end{proof}

\begin{proof}[Proof of Theorem~\ref{Thm:se-intro}] The result follows from Theorem~\ref{se} and Proposition~\ref{prop:Dini-k-vis}.
\end{proof}

\subsection{Localization of intrinsic distances}
From Theorem \ref{kp} and \eqref{roy}, we have
\begin{cor}
\label{loc}
Under the hypotheses of Theorem \ref{kp}, there exists
a neighborhood $V \subset \subset U$ of $p$ and a constant $c>0$ such that
$$\k_D(z;X)\ge (1-c\d_D(z))\k_{D_U}(z;X),\quad z\in D_V,\ X\in\C^n.$$
\end{cor}

The estimate in Corollary \ref{loc} is in the spirit of \cite[Theorem 2.1]{FR}.

Note that the hypotheses of Theorem \ref{kp} hold if $D\cap U$ is a bounded convex domain
with visibility property (or, more generally, if $\{p,q\}$ has visible geodesics for any
$q\in\partial D\cap U$). So, Corollary  \ref{loc} substantially generalizes the localization property
given by [Theorem 3.2, LW].

Assuming Dini-smooth regularity of $\partial D$ near $p,$ one may bootstrap the previous results to
obtain a localization for $k_D$
which is inspired by \cite[Theorem 1.4]{LW}: let $z,w$ be two points tending in the Euclidean sense
to a boundary point $p$; even though their Kobayashi distance might tend to infinity,
the difference between their Kobayashi distances
with respect to the local domain $D\cap U$  and the global domain $D$ tends to $0$ when $z,w \to p$:

\begin{prop}
\label{dini}
Let $D$ be a bounded domain in $\C^n$ with Dini-smooth boundary near a  point
$p\in\partial D.$ Assume that there exists a neighborhood $U$ of $p$ such that
$p$ is a $k$-point for $D\cap U$. Then
$$\lim_{z,w\to p}(k_{D\cap U}(z,w)-k_D(z,w))=0.$$
\end{prop}

\begin{proof}
Again, as a matter of notation, if $D, U$ are open sets, we set $D_U=D\cap U$. Since $\partial D$ is Dini-smooth near $p,$
by \cite[Theorem 7]{NA}, we may find a neighborhood $W\subset\subset U$ of $p$ such that
\begin{equation}
\label{naest}
k_{D_U}(z,w)\le\log\left(1+\frac{2\|z-w\|}{\sqrt{\d_D(z)\d_D(w)}}\right),
\quad z,w\in D_W.
\end{equation}
On the other hand, by Theorem \ref{kp}, $p$ is a $k$-point for $D$. Fix $\var\in(0,1).$
Then we may shrink $W$ (if necessary) such that $\mbox{diam}\ W<\var,$ \eqref{loci} holds for $c=1,$
and the following is true for another neighborhood $W'\subset W$ of $p$:
if $\g:[0,1]$ is a piecewise $\mathcal C^1$-curve with $\g(0)=z\in D_{W'},$ $\g(L)=w\in D_{W'},$ and
$\int_0^L\k_D(\g(t);\g'(t))dt<k_D(\gamma(0),\gamma(1))+\var,$
then $\g([0,L])\subset W.$ Then for $\s$ as in Lemma \ref{abs} (with
$L'=L$ and $w'=w$), using \eqref{naest}, we get that $\d(\s(t))<{f(t)},$ where
$f(t)=\varepsilon$ if $|t-t_0|<1+\var$ and $f(t)=\frac{2\var}{e^{|t-t_0|-1-\var}-1}$ otherwise. This implies that
$$k_{D_U}(z,w)<k_D(z,w)+\var+C\var,\quad z,w\in D_{W'}$$
for some constant $C>0$, and we are done.
\end{proof}

Now we are in good shape to prove Theorem \ref{visloc}:

\begin{proof}[Proof of Theorem \ref{visloc}]

Note that for any $a\in\overline{D_V}$ there exist a neighborhood $W_a$ and a constant $C_a$ 
such that if $\gamma : [0,1] \longrightarrow D_{W_a}$ is a piecewise $\mathcal C^1$ curve,
then its Kobayashi-Royden length verifies
\begin{equation}
\label{obs}
L_{\k_{D_U}}(\gamma)\le L_{\k_D}+C_a.
\end{equation}
This is trivial if $a\notin\partial D.$ If $a\in\partial D,$ then, by Proposition \ref{kpteq} (i), 
$p$ is a $k$-point for $D_U$ and \eqref{obs} follows by the proof of Theorem \ref{kp}.

Assume the theorem fails. Then we may find sequences of points $z_k\to p\in\overline D_V$ and 
$w_k\to q\in \overline D_V$ such that
$$k_{D_U}(z_k,w_k)-k_D(z_k,w_k)\to\infty.$$

If $p\neq q$, we may choose the neighborhoods so that $W_p\cap W_q=\emptyset$.
Restricting attention to tails of the sequences, we may assume that $z_k\in W_p$ and $w_k\in W_q$, for all $k$.

Let now $\gamma_k : [0,1] \longrightarrow D$ be a piecewise $\mathcal C^1$ curve with 
$\gamma_k(0)=z_k,$ $\gamma_k(1)=w_k$ and 
$L_{\k_D}(\gamma_k) \le k_D(z,w)+1$. 

If $p=q$ and $\gamma_k([0,1])\subset W_p,$ then
\begin{equation}
\label{singlept}
k_{D_U}(z_k,w_k)\le k_D(z_k,w_k)+C_p+1.
\end{equation}
%and this contradicts our assumtion.
Otherwise, then $\gamma_k([0,1])\not \subset W_p \cup W_q$, 
so set $r_k:= \inf\{r: \gamma_k(r) \notin D_{W_p}\}$, $z_k':= \gamma(r_k)$, and 
$s_k:= \sup\{s: \gamma(s) \notin D_{W_q}\}$, $w_k':= \gamma(s_k).$
It follows by \eqref{obs} that
\begin{equation}
\label{interm}
k_{D_U}(z_k,z_k') + k_{D_U}(w_k,w_k') \le k_D(z_k,w_k) +C_1,
\end{equation}
where $C_1=C_p+C_q+1.$

Choose now a base point $o\in D_U$.  The visibility property of $D_U$, the fact that
$z_k\to p\in\ W_p\not\ni z'_k$,
and Proposition \ref{necvis} imply that there exists a constant $C_2>0$
such that for all $k$ in this case, $k_{D_U}(z_k,z'_k) \ge k_{D_U}(z_k,o) + k_{D_U}(o,z'_k)-C_2$.
Analogously, enlarging $C_2$ if needed, since $w_k\to q\in\ W_q\not\ni w'_k,$
$k_{D_U}(w_k,w'_k) \ge k_{D_U}(w_k,o)+ k_{D_U}(o,w'_k) -C_2$.

Using \eqref{interm}, we obtain that
\begin{multline*}
k_D(z_k,w_k) +C_1\ge k_{D_U}(z_k,z'_k) + k_{D_U}(w_k,w'_k)  \\
\ge k_{D_U}(z_k,o) + k_{D_U}(o,z'_k)-C_2 + k_{D_U}(w_k,o)+ k_{D_U}(o,w'_k) -C_2 \\
\ge  k_{D_U}(z_k,w_k) + k_{D_U}(z'_k,w'_k) -2C_2 \ge k_{D_U}(z_k,w_k)  -2C_2,
\end{multline*}
so
\begin{equation}
\label{bound}
k_{D_U}(z_k,w_k) \le k_D(z_k,w_k) +C_1 +2C_2,
\end{equation}
Since for any $k$ either \eqref{singlept} or \eqref{bound} holds,  we reach a contradiction.
\end{proof}

{}

\end{document}